\documentclass[12pt]{article}
\usepackage{a4}
\usepackage{latexsym}
\usepackage{amsfonts}
\usepackage{amsmath, amssymb, amsthm}
\usepackage{graphicx,xcolor}
\usepackage{enumerate}
\usepackage[stable]{footmisc}

\newtheorem{theorem}{Theorem}%[section]
\newtheorem{conjecture}[theorem]{Conjecture}
\newtheorem{corollary}[theorem]{Corollary}

\newtheorem{lemma}[theorem]{Lemma}

\newtheorem{observation}[theorem]{Observation}
\newcommand{\ZZ}{\mathbb{Z}}
\newcommand{\scon}{\mathcal S}
\newcommand{\pcon}{\mathcal P}

\newcommand{\Setx}[1]{\left\{#1\right\}}

\newcommand{\depth}[1]{\mathop{\mathrm{dp}}(#1)}
\newcommand{\fig}[1]{\includegraphics[page=#1]{SP-figs}}
\newcommand{\hf}{\hspace*{0mm}\hspace{\fill}\hspace*{0mm}}

\title{\textbf{Nowhere-zero flows in signed series-parallel graphs}}
\author{%
  Tom\'{a}\v{s} Kaiser$\,^{1,2}$%
  \and Edita Rollov\'{a}$\,^{1,3}$%
}%
\date{}

\begin{document}
\maketitle

\footnotetext[1]{Department of Mathematics and European Centre of Excellence
  NTIS---New Technologies for Information Society, University of West
  Bohemia, Univerzitn\'{\i}~8, 306~14~Plze\v{n}, Czech
  Republic. Supported by project GA14-19503S of the Czech Science
  Foundation.}%
\footnotetext[2]{Affiliated with the Institute for Theoretical Computer
  Science (CE-ITI). E-mail: \texttt{kaisert@kma.zcu.cz}.}%
\footnotetext[3]{Supported by the project NEXLIZ --- CZ.1.07/2.3.00/30.0038,
  which is co-financed by the European Social Fund and the state
  budget of the Czech Republic. Partially supported by APVV, Project
  0223-10 (Slovakia). E-mail: \texttt{rollova@kma.zcu.cz}.}

\begin{abstract}
  Bouchet conjectured in 1983 that each signed graph {that} admits a
  nowhere-zero flow has a nowhere-zero $6$-flow. We prove that the
  conjecture is true for all signed series-parallel graphs. Unlike the
  unsigned case, the restriction to series-parallel graphs is
  nontrivial; in fact, the result is tight for infinitely many graphs.
\end{abstract}

\section{Introduction}
\label{sec:introduction}

A \emph{signed graph} $(G,\sigma)$ is a graph $G$ together with a
\emph{signature}, a mapping $\sigma:\ E(G) \to \{+1,-1\}$, that
assigns each edge with a sign. The graph $G$ is called the
\emph{underlying graph} of $(G,\sigma)$. In this work we focus on
signed graphs whose underlying graph is a \emph{series-parallel
  graph}, that is, a graph that can be obtained from copies of $K_2$
by iterated series and parallel connections.

A signed graph can be given an \emph{orientation} as follows. Viewing
each edge as composed of two half-edges, we orient each half-edge
independently; it is required that of the two half-edges of an edge
$e$, exactly one points to its endvertex if $e$ is positive, while
none or both of them point to their endvertices if $e$ is negative. A
\emph{nowhere-zero $k$-flow} $(D,\phi)$ on a signed graph $(G,\sigma)$
is an orientation $D$ of edges of $(G,\sigma)$ and a valuation $\phi$
of its arcs by non-zero integers whose absolute value is smaller than
$k$, such that for every vertex the sum of the incoming values (the
\emph{inflow}) is equal to the sum of the outgoing ones (the
\emph{outflow}). Graphs (signed or unsigned) admitting at least one
nowhere-zero {$k$-flow (for some $k$)} are called
\emph{flow-admissible}.

With respect to nowhere-zero flows, all-positive signed graphs (i.e.,
those with all edge signs positive) behave like ordinary unsigned
graphs. Thus, problems about nowhere-zero flows in signed graphs
include the celebrated 5-flow conjecture of Tutte~\cite{T54}:
\begin{conjecture}[Tutte]\label{conj:5-flow}
  Every {flow-admissible graph} has a nowhere-zero
  $5$-flow.
\end{conjecture}
While there are examples showing that the analogue of
Conjecture~\ref{conj:5-flow} is false for {general} signed graphs (as discussed
below), Bouchet~\cite{Bou} conjectured that things do not get much
worse:
\begin{conjecture}[Bouchet]\label{conj:bouchet}
  Every flow-admissible signed graph has a nowhere-zero $6$-flow.
\end{conjecture}
The best published partial result in the direction of Bouchet's
conjecture is a 30-flow theorem by Z\'yka~\cite{Zyka} from
1987. Recently, a 12-flow theorem was announced by DeVos~\cite{DeVos}.

An infinite family of signed graphs reaching the bound stated in
Conjecture~\ref{conj:bouchet} was found by Schubert and
Steffen~\cite{SS}. The smallest member of the family is shown in
Figure~\ref{fig:lower}. Interestingly, the underlying graphs of the
members of this family are series-parallel. This is in sharp contrast
with the situation in the unsigned case, where each flow-admissible
series-parallel graph trivially admits a nowhere-zero 3-flow. In this
paper, we concentrate on the family of signed series-parallel graphs
and prove the corresponding restriction of Bouchet's conjecture:
\begin{theorem}\label{t:main}
  Every flow-admissible signed series-parallel graph has a
  nowhere-zero $6$-flow.
\end{theorem}
The proof is given in Section~\ref{sec:proof}.

\begin{figure}
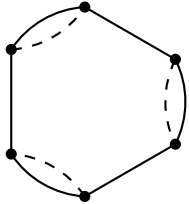

  \centering
  \fig1
  \caption{A signed series-parallel graph with flow number 6. }
  \label{fig:lower}
\end{figure}

%.....................................................................

\section{Preliminaries}

In this section, we review the necessary terminology. Our graphs may
contain parallel edges. The \emph{switching at a vertex $v$} of a
signed graph is the operation of inverting the signs on all the edges
incident with $v$. Two signed graphs are \emph{switching equivalent}
if one can be obtained from the other by a finite sequence of
switchings. Since switching does not affect the existence of a
nowhere-zero $k$-flow (for any $k$), we may treat switching equivalent
graphs as identical.

It is well known that an unsigned graph is flow-admissible if and only
if it is bridgeless. Before we state a corresponding characterisation
for signed graphs due to Bouchet~\cite{Bou}, we recall several basic
notions. A \emph{balanced cycle} is a cycle with an even number of
negative edges. An \emph{unbalanced cycle} is a cycle with an odd
number of negative edges. A signed graph is called \emph{unbalanced}
if it contains an unbalanced cycle. Otherwise it is \emph{balanced}. A
\emph{barbell} in a signed graph $G$ is the union of two edge-disjoint
unbalanced cycles $C_1$, $C_2$ and a path $P$ satisfying one of the
following properties:
\begin{itemize}
\item $C_1$ and $C_2$ are vertex-disjoint, $P$ is internally
  vertex-disjoint from $C_1\cup C_2$ and shares an endvertex with each
  $C_i$, or
\item $V(C_1)\cap V(C_2)$ consists of a single vertex $w$, and $P$ is
  the trivial path consisting of $w$.
\end{itemize}
A \emph{signed circuit} in $G$ is either a balanced cycle or a barbell
in $G$. With respect to flows, signed circuits are analogous to cycles
in unsigned graphs. The following characterisation theorem is due to
Bouchet~\cite[Proposition~3.1]{Bou}:

\begin{theorem}\label{t:admissible}
  A signed graph $G$ is flow-admissible if and only if each of its
  edges is contained in a signed circuit.
\end{theorem}

A \emph{two-terminal graph} $(G,s,t)$ consists of a graph $G$ together
with two distinguished vertices, the \emph{source terminal} $s$ and
the \emph{target terminal} $t$, where $s\neq t$. We abbreviate
$(G,s,t)$ to $G$.

Let $G_1,\dots,G_n$ be two-terminal graphs and $H$ their disjoint
union. The \emph{series connection} $\scon(G_1,\dots,G_n)$ of
$G_1,\dots,G_n$ is obtained from $H$ by identifying, for
$i=1,\dots,n-1$, the target terminal of $G_i$ with the source terminal
of $G_{i+1}$. By definition, the source and target terminal of
$\scon(G_1,\dots,G_n)$ is the source terminal of $G_1$ and the target
terminal of $G_n$, respectively.

The \emph{parallel connection} $\pcon(G_1,\dots,G_n)$ of
$G_1,\dots,G_n$ is obtained from $H$ by identifying their source
terminals and identifying their target terminals. The source terminal
of the resulting graph is the vertex obtained by the identification of
the source terminals of the graphs $G_i$, and similarly for the target
terminal.

A \emph{series-parallel graph} is a two-terminal graph obtained by a
sequence of series and parallel connections, starting with copies of
$K_2$ (with some choice of the terminals).

If $G_1,\dots,G_n$ ($n\geq 2$) are series-parallel graphs such that
$G$ is a series or a parallel connection of $G_1,\dots,G_n$ and $n$ is
maximum with this property, then we refer to the $G_i$ as \emph{parts}
of $G$. In addition, in the case of a series connection, $G_1$ and
$G_n$ are the \emph{endparts} of $G$. We say that $G'$ is a
\emph{piece} of $G$ if there is a sequence $G'=H_0,H_1,\dots,H_m=G$
such that for each $j=0,\dots,m-1$, $H_j$ is a part of $H_{j+1}$. In
particular, $G$ itself is a piece of $G$.

The \emph{replacement} of $G'$ by a series-parallel graph $H'$ in $G$
consists in removing all the edges and non-terminal vertices of $G'$
in $G$, adding $H'$ and identifying each of its terminals with the
corresponding terminal of $G'$ in $G$.

We introduce the following notation for small signed series-parallel
graphs: $K_2^+$ denotes the positive $K_2$, $K_2^-$ stands for the
negative $K_2$, and $D$ is the unbalanced 2-cycle.

We define the \emph{depth} $\depth G$ of a signed series-parallel
graph $G$ by letting $\depth{K_2^+} = \depth{K_2^-} = 0$, and
\begin{equation*}
  \depth G = 1 + \max_H\depth H,
\end{equation*}
where $H$ ranges over all parts of $G$.

The following observation is immediate from the definition:
\begin{observation}\label{obs:depth}
  If $0\leq k \leq \depth G$, then $G$ contains a piece of depth $k$.
\end{observation}

Another useful observation is the following one; it can be proved by
straightforward induction:
\begin{observation}\label{obs:two}
  Each non-terminal vertex of a series-parallel graph has at least two
  distinct neighbours.
\end{observation}

We will need the following lemma, proved in~\cite{Epp} by induction:
\begin{lemma}\label{l:change}
  If $e=xy$ is an edge in a $2$-connected series-parallel graph $G$,
  then $(G,x,y)$ is a series-parallel graph.
\end{lemma}

%.....................................................................

\section{Reduced graphs}
\label{sec:reduced}

We begin by using easy reductions to prove the following observation:

\begin{lemma}\label{l:reductions}
  Let $G$ be a counterexample to Theorem~\ref{t:main} of minimum
  size. Then $G$ has the following properties:
  \begin{enumerate}[(i)]
  \item if $G$ is of series type, then its endparts are unbalanced,
  \item the degree of each non-terminal vertex is at least three,
  \item if a terminal vertex has degree two, then it is contained in a
    2-cycle.
  \end{enumerate}
\end{lemma}
\begin{proof}
  We prove (i). Suppose that $G$ has a balanced endpart $H$; let $u$
  denote the terminal $u$ of $H$ that is a cutvertex of $G$, and let
  $H'$ be obtained from $G$ by removing $V(H)-\Setx{u}$. We show using
  Theorem~\ref{t:admissible} that $H'$ is flow-admissible. Suppose the
  contrary; then an edge $e$ of $H'$ is not contained in a signed
  circuit of $H'$. However, $G$ is flow-admissible, and it is not hard
  to see that any signed circuit of $G$ missing in $H'$ is a balanced
  circuit contained in $H$ (as $H$ is balanced and $u$ is a
  cut-vertex). This is a contradiction, so $H'$ is indeed
  flow-admissible. By the minimality of $G$, $H'$ admits a
  nowhere-zero 6-flow. Since any unsigned series-parallel graph admits
  a nowhere-zero 3-flow, so does the balanced signed graph
  $H$. Combining these two flows, we obtain a nowhere-zero 6-flow on
  $G$, a contradiction.

  Let us prove (ii). Let $u$ be a non-terminal vertex of degree 2, say
  incident with edges $e_1,e_2$. Switching at $u$ if necessary, we may
  assume that $e_1$ is positive. By Observation~\ref{obs:two}, the
  endvertices of $e_1,e_2$ different from $u$ are
  distinct. Contracting $e_1$, we therefore obtain a (loopless)
  series-parallel graph $G'$. Since the contraction of a positive edge
  preserves the existence of a nowhere-zero flow, $G'$ is
  flow-admissible and hence it has a nowhere-zero 6-flow by the
  minimality of $G$. This corresponds to a 6-flow $\psi$ on $G$,
  possibly with $\psi(e_1)=0$. However, since $u$ has degree 2, we
  have $|\psi(e_1)|=|\psi(e_2)|$, so $\psi$ is nowhere-zero.

  A similar argument works for terminal vertices of degree 2 with two
  distinct neighbours. This proves (iii).
\end{proof}

The following result will be useful in the proof of
Lemmas~\ref{l:parallel} and \ref{l:admissible} below.

\begin{lemma}
  \label{l:endparts}
  If a signed series-parallel graph $G'$ is of series type and has
  unbalanced endparts, then every edge of $G'$ is contained in a
  barbell.
\end{lemma}
\begin{proof}
  Let the endparts of $G'$ be denoted by $R_1$, $R_2$. For $i=1,2$,
  let $r_i$ be the terminal of $R_i$ that is not a terminal of
  $G$. Let $C_i$ ($i=1,2$) be an unbalanced cycle in $R_i$.

  Let $e$ be an edge of $G'$; we need to show that $e$ is contained in
  a barbell. If $e$ is not contained in an endpart, then $G'$ contains
  a path $P$ from $r_1$ to $r_2$ containing $e$. Extending $P$ to a
  path joining $C_1$ to $C_2$, we obtain a barbell in $G'$ containing
  $e$.

  We may therefore assume that $e$ is contained in $R_1$. Since $R_1$
  is unbalanced, $R_1$ is different from a signed $K_2$ and therefore
  2-connected. Choose two vertex-disjoint paths in $R_1$ connecting
  the endvertices of $e$ to $C_1$. Taking the union of these paths
  with $e$ and a suitable subpath of $C_1$, we obtain an unbalanced
  cycle $C'_1$ in $R_1$ containing $e$. Then the union of $C'_1$,
  $C_2$ and an appropriate extension of $P$ is a barbell in $G'$
  containing $e$.
\end{proof}

\begin{lemma}\label{l:parallel}
  If $G$ is a counterexample to Theorem~\ref{t:main} of minimum size,
  then $G$ contains no pair of parallel edges of the same sign.
\end{lemma}
\begin{proof}
  Suppose that $G$ contains parallel edges $e,f$ of the same sign, say
  both positive. By Theorem~\ref{t:admissible}, each edge of $G$ is
  contained in a signed circuit. Thus, each edge $e'$ of $G-e$ is
  clearly also contained in a signed circuit of $G-e$ unless $e'=f$.

  Assume first that $f$ is also contained in a signed circuit of
  $G-e$. Then $G-e$ is flow-admissible. By the minimality of $G$,
  $G-e$ admits a nowhere-zero 6-flow $\varphi$. Adding to $\varphi$ a
  suitable nowhere-zero 6-flow on the 2-cycle $C$ comprised of $e$ and
  $f$, we obtain a nowhere-zero 6-flow in $G$.

  It follows that $f$ is not contained in a signed circuit. Note that
  in this case, $G-e-f$ is flow-admissible. If we can show that each
  component of $G-e-f$ is series-parallel, then by the minimality of
  $G$, $G-e-f$ admits a nowhere-zero 6-flow, which is easily extended
  to $G$ using a suitable flow on the above 2-cycle $C$.

  It remains to prove that $G-e-f$ is comprised of series-parallel
  components. Let the terminals of $G$ be $u$ and $v$. Suppose first
  that $G-e$ is 2-connected. Then $G-e-f$ is connected. By
  Lemma~\ref{l:change}, $(G-e-f,x,y)$ is a series-parallel graph,
  where $x$ and $y$ are the endvertices of $f$.

  We can therefore assume that $G-e$ is not 2-connected. We claim that $G$ is not 2-connected; assume the contrary. Since, clearly,
  $G-e$ is different from $K_2$, there is a cutvertex of $G-e$
  separating the endvertices of $e$. These are, however, connected by
  the edge $f$, a contradiction.

  Since $G$ is not 2-connected, it has unbalanced endparts by
  Lemma~\ref{l:reductions}(i), and the same holds for $G-e$. By
  Lemma~\ref{l:endparts}, $f$ is contained in a barbell in
  $G-e$, a contradiction with the hypothesis that $f$ is not contained
  in a signed circuit of $G-e$. This concludes the proof.
\end{proof}
 
Let us call a graph \emph{reduced} if the degree of each of its
non-terminal vertices is at least 3, and there is no pair of parallel
edges of the same sign. Observe that each part (and hence, each piece)
of a reduced graph is reduced. It is easy to see that the only reduced
graph of depth 1 is the unbalanced 2-cycle $D$.

A \emph{string} is a series connection of copies of $K_2^+$ and $D$
where each non-terminal vertex is contained in a 2-cycle (see
Figure~\ref{fig:string}). Thus, any string is a reduced graph. A
string is \emph{nontrivial} if it contains more than two vertices.

\begin{figure}
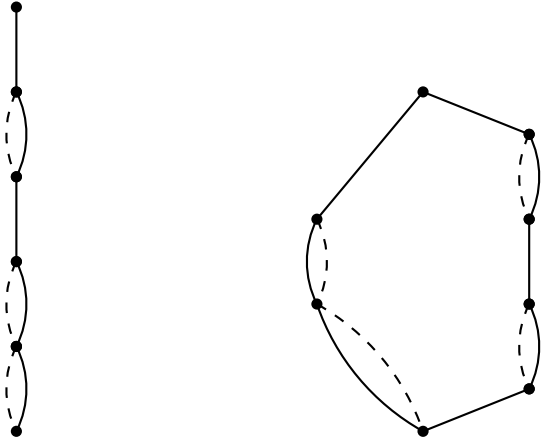

  \centering
  \hf\fig2\hf\fig3\hf
  \caption{A string (left) and a necklace (right). In this and the
    following figure, the source terminal is the topmost vertex and
    the target terminal is the lowermost one.}
  \label{fig:string}
\end{figure}

\begin{lemma}\label{l:equiv}
  Each reduced signed series-parallel graph of depth at most $2$ is
  switching equivalent to a string.
\end{lemma}
\begin{proof}
  Let $G$ be a graph satisfying the assumption. If the depth of $G$ is
  0 or 1, then $G$ is $K_2^+$, $K_2^-$ or $D$ and the assertion
  holds. Assume that the depth of $G$ is 2. Since the only reduced
  graph of depth 1 is $D$, $G$ is necessarily of series type; since
  each of its parts is reduced, $G$ is a series connection of copies
  of $K_2^+$, $K_2^-$ and $D$. Switching at the target terminal of
  each $K_2^-$, we obtain a string.
\end{proof}

A \emph{necklace} is a signed series-parallel graph obtained as the
parallel connection of two strings, at least one of which is
nontrivial (see Figure~\ref{fig:string}).

\begin{lemma}\label{l:contains}
  Each reduced signed series-parallel signed graph of depth at least
  $3$ contains a piece switching equivalent to a necklace.
\end{lemma}
\begin{proof}
  Let $G$ be a signed series-parallel graph of depth at least 3. By
  Observation~\ref{obs:depth}, $G$ contains a piece of depth 3 which
  is necessarily reduced. We may therefore assume that the depth of
  $G$ is equal to 3. Since one of its parts has depth 2 and is
  reduced, it is of series type, so $G$ itself is of parallel
  type. Lemma~\ref{l:equiv} implies that $G$ is a parallel connection
  of graphs switching equivalent to a string. Let $H$ be the parallel
  connection of two of these graphs, say $H_1$ and $H_2$, where the
  depth of $H_1$ equals 2. Note that $H_1$ has more than two
  vertices. We show that $H$ is switching equivalent to a
  necklace. First, if $H_2$ is a $K_2^-$, then we perform a switch at
  one of the terminals to change its sign. Each of the remaining
  negative edges has an endvertex contained in an unbalanced 2-cycle;
  we switch at each such endvertex to obtain a necklace.
\end{proof}

%.....................................................................

\section{Pseudoflows}
\label{sec:pseudo}

For the proof of Theorem~\ref{t:main}, we utilise the concept of
\emph{pseudoflow} in a signed series-parallel graph $H$, defined just
as a nowhere-zero $6$-flow in $H$, except that at each terminal, the
inflow is not required to equal the outflow. (In particular,
pseudoflows are nowhere-zero by definition.) Let
$I_5=\{-5,-4,\ldots,5\}$. A pseudoflow in $H$ is an
\emph{$(a,b)$-pseudoflow} (where $a,b\in I_5$) if the outflow at the
source terminal equals $a$ and the inflow at the target terminal
equals $b$. An example of a pseudoflow is shown in
Figure~\ref{fig:pseudo}. As another example, note that $K_2^+$ admits
an $(a,b)$-pseudoflow if and only if $a=b\neq 0$ and $a\in I_5$.

\begin{figure}
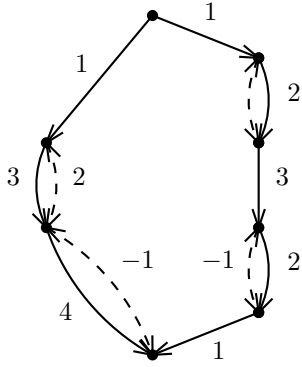

  \centering 
  \fig4
  \caption{A $(2,4)$-pseudoflow in a signed series-parallel graph.}
  \label{fig:pseudo}
\end{figure}

We make a couple of observations related to pseudoflows. An
$(a,b)$-pseudoflow can only exist if $a$ and $b$ have the same
parity. A $(0,0)$-pseudoflow coincides with a nowhere-zero
$6$-flow. Furthermore, if the source terminal of $H$ has degree 1,
then $H$ admits no $(0,b)$-pseudoflow for any $b$. Based on the last
observation, let us say that the pair $(a,b)$ is \emph{valid} for
$H=(H,s,t)$ if either $a \neq 0$ or $d(s) \geq 2$, and at the same time
$b \neq 0$ or $d(t) \geq 2$.

\begin{observation}\label{obs:digon}
  Let $a,b\in\ZZ$ such that $a\equiv b \pmod 2$. If the unbalanced
  $2$-cycle $D$ admits an $(a,b)$-pseudoflow, then $a\neq\pm b$. In the
  converse direction, if $a\neq\pm b$ and $a,b\in I_5$, then $D$
  admits an $(a,b)$-pseudoflow.
\end{observation}

The following lemma provides us with information on the types of
pseudoflows that exist in strings. For a graph $G$, we define
$\beta(G)$ as the number of distinct 2-cycles in $G$. Note that
$\beta(G)\geq 1$ for any nontrivial string $G$.

\begin{lemma}\label{l:string}
  Let $G$ be a nontrivial string. Let $a,b\in I_5$ be integers such
  that $a\equiv b\pmod 2$ and $(a,b)$ is valid for $G$. Then $G$
  admits an $(a,b)$-pseudoflow if one of the following conditions
  holds:
  \begin{enumerate}[\quad(a)]
  \item $\beta(G)$ is odd and $a\neq\pm b$,
  \item $\beta(G)$ is even and $a=\pm b$,
  \item $\beta(G) \geq 2$ and either $a$ is odd or $a=0$ or $b=0$.
  \end{enumerate}
\end{lemma}
\begin{proof}
  Let $n=\beta(G)$ and let the 2-cycles of $G$ be
  $D_1,\dots,D_n$. Orient the positive edges of $G$ so as to obtain a
  directed path from the source to the target terminal. We claim that
  it suffices to find a sequence of numbers $c_1,\dots,c_{n+1}\in
  I_5$, all of the same parity, such that $a=c_1$, $b=c_{n+1}$, $c_i
  \neq \pm c_{i+1}$ for all $i=1,\dots,n$, and $c_i\neq 0$ unless
  $i\in\Setx{1,n+1}$. Given such a sequence, we construct an
  $(a,b)$-pseudoflow on $G$ as follows:
  \begin{itemize}
  \item using Observation~\ref{obs:digon}, we find a
    $(c_i,c_{i+1})$-pseudoflow on each $D_i$, where $1\leq i \leq n$,
  \item we assign flow value $c_{i+1}$ to a bridge joining $D_i$ to
    $D_{i+1}$ if there is one ($1\leq i \leq n$), with respect to the
    fixed orientation,
  \item a bridge incident with a source (target) terminal, if such a
    bridge exists, will be assigned value $a=c_1$ ($b=c_{n+1}$,
    respectively) with respect to the fixed orientation.
  \end{itemize}
  It is not hard to see that this procedure defines an
  $(a,b)$-pseudoflow on $G$. Note that the values associated to any
  bridges incident with terminals are nonzero since the pair $(a,b)$
  is valid for $G$.

  To find a sequence as above, we consider the possible cases one by
  one. In cases (a) and (b), we just take the alternating sequence
  $a,b,a,b,\dots$ of length $n+1$. Consider case (c) and assume first
  that $a$ is odd. Let $d$ be an odd element of $I_5$ such that $\pm
  a\neq d \neq\pm b$. If $n$ is odd, the alternating sequence
  $a,b,\dots$ of length $n+1$ ends with $b$ as required; otherwise, we
  insert $d$ after its first element and delete the last element,
  again obtaining a sequence with the required property.

  To finish the discussion of case (c), assume that $a=0$. Let $d$ be
  an even element of $I_5$ such that $0\neq d \neq\pm b$. Depending on
  the parity of $n$, we take either the alternating sequence
  $0,b,\dots$ of length $n+1$, or the sequence obtained by inserting
  $d$ after the first $0$ and dropping the last element. The case
  $b=0$ is symmetric.
\end{proof}

We use Lemma~\ref{l:string} to obtain a similar result for necklaces:
\begin{lemma}\label{l:necklace}
  Let $G$ be a necklace and $a,b\in I_5$ integers such that $a\equiv b
  \pmod 2$. Then $G$ admits an $(a,b)$-pseudoflow if either $a\neq \pm
  b$, or $a=b=0$ and $\beta(G) \geq 2$.
\end{lemma}
\begin{proof}
  Suppose that $G$ is {a} parallel connection of strings $G_1$ and
  $G_2$. Label the strings in such a way that the following conditions
  hold if possible, in the order of precedence:
  \begin{itemize}
  \item $\beta(G_2) = 0$,
  \item $\beta(G_2)$ is odd.
  \end{itemize}

  We say that $G$ is of \emph{type I} if $\beta(G_2)$ is
  odd.  If $\beta(G_2)$ is
  even, then $G$ is of \emph{type II}.

{Suppose first that $a=b=0$ and $\beta(G)\geq 2$.} If both $\beta(G_1)$ and $\beta(G_2)$ are nonzero,{
  then by Lemma~\ref{l:string} (a) and (c), there is a $(1,3)$-pseudoflow in
  $G_1$ and a} $(-1,-3)$-pseudoflow in $G_2$; their sum is the required $(0,0)$-pseudoflow in $G$. Suppose then that $\beta(G_2) = 0$, which
  means that $\beta(G_1) \geq 2$. By Lemma~\ref{l:string} {(c)}, $G_1$
  admits a $(1,1)$-pseudoflow $f_1$. Since $G_2$ is necessarily the graph $K_2^+$, it admits a $(-1,-1)$-pseudoflow, which we again sum with $f_1$ to obtain a $(0,0)$-pseudoflow in $G$.

{For the rest of the proof suppose} that $a,b\in I_5$ are integers of the same parity such that $a\neq\pm b$. {The desired $(a,b)$-pseudoflow on $G$ will be constructed as a sum of an $(a_1,b_1)$-pseudoflow on $G_1$ and an $(a_2,b_2)$-pseudoflow on $G_2$ for suitable $a_1,b_1,a_2$ and $b_2$.}

{Let us consider possible pseudoflows in $G_1$. Let $a_1,b_1\in I_5$ such
  that $a_1\equiv b_1\pmod 2$. By the choice of $G_2$ and the fact that at least one string of a necklace is nontrivial, $\beta(G_1)\geq 1$. If $\beta(G_1)=1$, by Lemma~\ref{l:string} (a), $G_1$ admits an $(a_1,b_1)$-pseudoflow if $a_1\neq\pm b_1$. If $\beta(G_1)\geq 2$, then, by Lemma~\ref{l:string} (c), $G_1$ admits an $(a_1,b_1)$-pseudoflow if $a_1$ and $b_1$ are odd. In summary, regardless of the type of $G_2$,} $G_1$ admits an $(a_1,b_1)$-pseudoflow if 
  \begin{equation}
    \text{$a_1$ and $b_1$ are odd and $a_1\neq\pm
      b_1$.} \label{eq:first}
  \end{equation}
  Next, we consider pseudoflows in $G_2$. Let $a_2,b_2\in I_5$ such
  that $a_2\equiv b_2\pmod 2$. By Lemma~\ref{l:string} {(a) and (b)}, $G_2$ admits
  an $(a_2,b_2)$-pseudoflow if 
  \begin{equation}
    \label{eq:second}
    \text{either $G$ is of type I and $a_2\neq\pm
      b_2$, or $G$ is of type II and $a_2=\pm b_2$.}
  \end{equation}

  {For each possible $(a,b)$ we now} exhibit a
  choice of $(a_1,b_1)$ and $(a_2,b_2)$ satisfying the conditions
  \eqref{eq:first} and \eqref{eq:second}, respectively, and such that
  $a_1+a_2=a$, $b_1+b_2=b$. The $(a,b)$-pseudoflow {in $G$} will be the
  sum of an $(a_1,b_1)$-pseudoflow in $G_1$ and an
  $(a_2,b_2)$-pseudoflow in $G_2$.

  The choices of $(a_1,b_1)$ and $(a_2,b_2)$ are given in
  Table~\ref{tab:necklaces}. We assume, without loss of generality,
  that $|a|\leq |b|$. By inverting all signs if necessary, we may
  further assume that $a \geq 0$, and if $a=0$, then $b\geq 0$.

  For example, if $(a,b) = (0,2)$ and $G$ is of type II, then the
  table suggests taking $(a_1,b_1) = (1,3)$ and $(a_2,b_2)=(-1,-1)$,
  in accordance with conditions~\eqref{eq:first} and
  \eqref{eq:second}. The rest of the proof is a routine inspection of
  the table.

  \begin{table}
    \centering
    \begin{tabular}{c|cc|cc}
      & type I & & type II & \\
      $a$ & $a_1$ & $a_2$ & $a_1$ & $a_2$\\
      $b$ & $b_1$ & $b_2$ & $b_1$ & $b_2$\\\hline
      $0$ & 1 & $-1$ & 1 & $-1$\\
      $2$ & 5 & $-3$ & 3 & $-1$\\\hline
      $0$ & 3 & $-3$ & 1 & $-1$\\
      $4$ & 5 & $-1$ & 5 & $-1$\\\hline
      $2$ & $-1$ & 3 & 1 & 1\\
      $4$ & 5 & $-1$ & 3 & 1\\\hline
      $2$ & $-1$ & 3 & 1 & 1\\
      $-4$ & $-5$ & 1 & $-5$ & 1\\\hline
      $1$ & 3 & $-2$ & 3 & $-2$\\
      $3$ & $-1$ & 4 & 5 & $-2$\\\hline
      $1$ & 3 & $-2$ & 3 & $-2$\\
      $-3$ & 1 & $-4$ & $-1$ & $-2$\\\hline
      $1$ & 3 & $-2$ & $-1$ & 2\\
      $5$ & 1 & 4 & 3 & 2\\\hline
      $1$ & 3 & $-2$ & 5 & $-4$\\
      $-5$ & $-1$ & $-4$ & $-1$ & $-4$\\\hline
      $3$ & 5 & $-2$ & 1 & 2\\
      $5$ & 1 & 4 & 3 & 2\\\hline
      $3$ & 5 & $-2$ & 5 & $-2$\\
      $-5$ & $-1$ & $-4$ & $-3$ & $-2$
    \end{tabular}
    \caption{The pairs $(a_1,b_1)$ and $(a_2,b_2)$ (the first and second
      column in each field, respectively) 
      for each possible choice of $(a,b)$ and type of the necklace in
      the proof of Lemma~\ref{l:necklace}. The case $(a,b)=(0,0)$ is
      discussed separately.}
    \label{tab:necklaces}
  \end{table}
\end{proof}

\begin{corollary}\label{cor:flow}
  Every flow-admissible string or necklace admits a $(0,0)$-pseudoflow
  (that is, a nowhere-zero $6$-flow).
\end{corollary}

%.....................................................................

\section{Proof of Theorem~\ref{t:main}}
\label{sec:proof}

Let $G$ be a counterexample to Theorem~\ref{t:main} with minimum
number of edges and, subject to this condition, maximum number of
vertices. By Lemmas~\ref{l:reductions} and~\ref{l:parallel}, $G$ is
reduced, and by Corollary~\ref{cor:flow}, its depth is at least $3$.
Using Lemma~\ref{l:contains}, we may assume that $G$ contains a piece
$H$ that is a necklace. Furthermore, by Corollary~\ref{cor:flow},
$G\neq H$.

We choose $H$ in such a way that $\beta(H)$ is minimized. Recall that
$\beta(H)$ is the number of 2-cycles in $H$, and that $\beta(H)\geq
1$.

\textbf{Case A:} $H$ is not an endpart of $G$. 

We replace $H$ with $D'=\scon(K_2^+,D,K_2^+)$ in $G$, where $D$ is the
unbalanced 2-cycle. Let us call the resulting graph $G'$. The
following lemma provides the last missing piece in our argument.

\begin{lemma}\label{l:admissible}
  If $G$ contains an unbalanced cycle edge-disjoint from $H$, then
  $G'$ is flow-admissible.
\end{lemma}
\begin{proof}
  Let $u'$ and $v'$ be the source and target terminal, respectively,
  of the necklace $H$ (as well as of the graph $D'$). Let $D_0$ denote
  the 2-cycle in $D'$. Consider an arbitrary edge $e$ of $G'$. We need
  to show that $e$ is contained in a signed circuit of $G'$. Suppose
  the contrary.

  \textbf{Case 1:} $e \notin E(D')$.

  We first observe that $G$ contains no $u'v'$-path that contains $e$
  and is vertex-disjoint from $H$ except for its endvertices (let us
  call such a path an \emph{$e$-detour}). Indeed, combining such a
  path with one of the two $u'v'$-paths in $D'$ would provide us with
  a balanced cycle {of $G'$} containing $e$.

  In particular, each cycle of $G$ containing $e$ is edge-disjoint
  from $H$. Since $e$ is not contained in any signed circuit of $G'$,
  any such cycle must be unbalanced. On the other hand, since $G$ is
  flow-admissible, $e$ is contained in a signed circuit $B$ of $G$,
  which must therefore be a barbell. {Clearly, $B$ is not
    edge-disjoint from $H$, for otherwise $B$ is a signed circuit
    containing $e$ in $G'$.} Let $A_1$ and $A_2$ be the unbalanced
  cycles in $B$ and let $Q$ be the path connecting them.
  
{If $e$ belongs to an unbalanced cycle of $B$, say $A_1$, then $A_1$ is edge-disjoint from $H$ in $G$, and thus also from $D'$ in $G'$. Since $G'$ is connected, there exists a path connecting $A_1$ and $D_0$ and therefore also a barbell of $G'$ containing $e$, a contradiction. Hence we can assume that $e\in Q$.}

Note first that $A_1$ and $A_2$ are not both edge-disjoint from $H$,
  for otherwise {$Q$} contains both $u'$ and $v'$, and replacing the
  part of {$Q$} inside $H$ with a $u'v'$-path in $D'$ yields a barbell in $G'$ containing $e$.

  Suppose then that $A_1$ contains an edge of $H$. 
  %Note that $A_1$ does not contain $e$, for otherwise a subpath of $A_1$ would be an $e$-detour.
  We claim that $A_2$ is edge-disjoint from $H$. For the sake of a
  contradiction, assume that $A_2$ contains an edge of $H$. If both
  $A_1$ and $A_2$ were contained in $H$, then a subpath of $Q$ would
  be an $e$-detour. On the other hand, at least one $A_i$ has to be
  contained in $H$, for otherwise they both contain $u'$ and $v'$,
  violating the definition of a barbell.

  We may thus assume that $A_1$ is contained in $H$ and $A_2$ is
  not. Since $A_2$ contains both terminals of $H$, we have $Q\subseteq
  H$, {which is a contradiction, since $e\notin E(D')$. Therefore $A_2$ is edge-disjoint from $H$ as claimed. Let us} choose a shortest path $P$ in $Q\cup A_1$ connecting
  $A_2$ to a terminal of $H$; the union of $P$, $A_2$, $D_0$ and an
  edge of $D'$ connecting $D_0$ to an endvertex of $P$ is then a
  barbell in $G'$ containing $e$. This finishes the discussion of
  Case~1.

  \textbf{Case 2:} $e \in E(D')$.

  We show that $G$ is of parallel type. Otherwise, it would be of
  series type, its endparts would be unbalanced
  (Lemma~\ref{l:reductions}(i)) and by Lemma~\ref{l:endparts}, $e$
  would be contained in a barbell, which is {a contradiction}.

  The graph $G'$ is also of parallel type. This is clear if $H$ is a
  (proper) subgraph of one of the parts of $G$. Otherwise, the two
  strings forming $H$ are parts of $G$, and by the assumption that $G$
  contains an unbalanced cycle {edge-disjoint from $H$}, there are
  some more parts. Then $G'$ is the parallel connection of these parts
  with $D'$.

  By symmetry, we may assume that $e$ is not incident with the
  terminal $v'$ of $D'$. Let $A$ be an unbalanced cycle in $G$
  edge-disjoint from $D'$. Since $G'$ is 2-connected, it contains a
  path $R$ joining $u'$ to $A$ and avoiding $v'$. The union of $D_0$,
  $R$, $A$ and the edge connecting $D_0$ to $u'$ is a barbell
  containing $e$. This concludes the proof. 
\end{proof}

We claim that there is indeed an unbalanced cycle of $G$ that is
edge-disjoint from $H$, as required in Lemma~\ref{l:admissible}. {This is clear if $G$ is of series type, because every endpart of $G$ is unbalanced (Lemma~\ref{l:reductions}(i)).}

Let $G$ be a parallel connection of its parts $H_1,\dots,H_k$. The
necklace $H$ is either a union of two strings $H_i\cup H_j$ for some
$i$ and $j$, or it is a proper subgraph of some $H_i$. Suppose first
the former --- say, $H=H_1\cup H_2$. Then $k\geq 3$, because $G\neq
H$. If there exists $i\in\{3,\ldots,k\}$ such that $H_i$ contains an
unbalanced cycle, we are done.  Since $G$ is reduced, each $H_i$ is
reduced as well, and thus every $H_i$ is a signed $K_2$. If there
exist $H_i$ and $H_j$ ($i,j\in\{3,\ldots,k\}$) such that they have
opposite signs, then $H_i\cup H_j$ is the sought unbalanced cycle
edge-disjoint from $H$. We conclude that $k=3$, because $G$ is reduced
and it does not contain parallel edges of the same sign. If neither
$H_1$ nor $H_2$ is a $K_2^+$, then $H_1\cup H_3$ is a necklace of $G$
with $\beta(H_1\cup H_3)<\beta(H)$, which is a contradiction with the
choice of $H$. On the other hand, if $H$ contains a string that is a
$K_2^+$, then the string forms an unbalanced 2-cycle with $H_3$ (as $G$ is
reduced), and $G$ is a necklace, which is {a contradiction with}
Corollary~\ref{cor:flow}.

Suppose now that $H$ is a proper subgraph of one of the parts of $G$,
say $H_1$. By a similar argument as above we conclude that $k=2$ and
$H_2$ is a signed $K_2$. Moreover $H_1$ is a series connection of
(reduced) graphs $H_{11},\ldots, H_{1q}$ for some $q$. If any of the
graphs $H_{1i}$ that does not contain $H$ contains an unbalanced
cycle, we are done. Therefore each $H_{1i}$ that does not contain $H$
is a signed $K_2$. Since $H$ is a proper subgraph of $H_1$, we
conclude that at least one of the terminals of $G$ must be of degree
2. But the terminal is not contained in an unbalanced 2-cycle, which
is a contradiction with Lemma~\ref{l:reductions} (c). This proves the
claim that $G$ satisfies the hypothesis of Lemma~\ref{l:admissible}
and the graph $G'$ is flow-admissible.

The graph $D'$, used to obtain $G'$, has
four edges. All necklaces have at least four edges, with
$\pcon(K_2^+,\scon(K_2^+,D))$ and its mirror image being the only ones
with exactly four. These necklaces, however, have one vertex fewer
than $D'$, so the choice of $G$ implies that $G'$ admits a
nowhere-zero $6$-flow $\varphi$.

The restriction of $\varphi$ to $D_0$ (the 2-cycle in $D'$) is an
$(a,b)$-pseudoflow for some $a$ and $b$. By
Observation~\ref{obs:digon}, $a\equiv b\pmod 2$ and $a\neq\pm
b$. Furthermore, since the values of $\varphi$ on the edges of
$D'-E(D_0)$ are $\pm a$ and $\pm b$, we find that $a,b$ are nonzero
elements of $I_5$. By Lemma~\ref{l:necklace}, $H$ admits an
$(a,b)$-pseudoflow. Combining it with a restriction of $\varphi$, we
obtain a nowhere-zero $6$-flow in $G$, contradicting the hypothesis.

\textbf{Case B:} $H$ is an endpart of $G$. 

The argument is similar {to the argument of Case A}, except that
we use the graph $\scon(K_2^+,D)$ or $\scon(D,K_2^+)$ in place of $D'$
(so as to obtain a flow-admissible graph) and invoke an analogue of
Lemma~\ref{l:admissible}. We arrive at a similar contradiction, which
concludes the proof.

%.....................................................................

% \section*{Acknowledgments}

% We gratefully acknowledge support from project GA14-19503S of the
% Czech Science Foundation, as well as partial support for the second
% author by APVV, Project 0223-10 (Slovakia). The second author was
% supported by the project NEXLIZ --- CZ.1.07/2.3.00/30.0038, which is
% co-financed by the European Social Fund and the state budget of the
% Czech Republic.

%%%%%%%%%%%%%%%%%%%%%%%%%%%%%%%%%%%%%%%%%%%%%%%%%%%%%%%%%%%%%%%%%%%%%%

\end{document}